\documentclass[11pt]{amsart}
\usepackage{fullpage,hyperref,subcaption}
\usepackage{amssymb,amsmath,amsthm,mathtools,extpfeil}
\usepackage[all,cmtip]{xy}
\usepackage{tikz}
\usepackage{rotating,enumerate}
\newtheorem{thm}[equation]{Theorem}
\newtheorem*{thm*}{Theorem}
\newtheorem{lem}[equation]{Lemma}
\newtheorem{prop}[equation]{Proposition}
\newtheorem{cor}[equation]{Corollary}

\theoremstyle{definition}

\newtheorem{rmk}[equation]{Remark}

\numberwithin{equation}{section}

\newtheorem{innercustomthm}{Theorem}
\newenvironment{refthm}[1]
  {\renewcommand\theinnercustomthm{#1}\innercustomthm}
  {\endinnercustomthm}

\DeclareMathOperator{\id}{id}
\DeclareMathOperator{\st}{st}

\DeclareMathOperator{\lk}{lk}
\DeclareMathOperator{\img}{im}

\DeclareMathOperator{\vol}{vol}
\DeclareMathOperator{\coker}{coker}

\DeclareMathOperator{\Gr}{Gr}

\DeclareMathOperator{\Lip}{Lip}
\newcommand{\ph}{\varphi}

\begin{document}
\title{Algorithmic aspects of immersibility and embeddability}
\author{Fedor Manin and Shmuel Weinberger}

\begin{abstract}
  We analyze an algorithmic question about immersion theory: for which $m$, $n$,
  and $CAT=\mathbf{Diff}$ or $\mathbf{PL}$ is the question of whether an
  $m$-dimensional $CAT$-manifold is immersible in $\mathbb{R}^n$ decidable?  As
  a corollary, we show that the smooth embeddability of an $m$-manifold with
  boundary in $\mathbb{R}^n$ is undecidable when $n-m$ is even and
  $11m \geq 10n+1$.
\end{abstract}
\maketitle

\section{Introduction}

The problem of classifying immersions of one smooth manifold $M$ in another $N$
was, in a sense, solved by Smale and Hirsch \cite{Smale} \cite{Hirsch}, who
reduced the question to one in homotopy theory.  This is now viewed as an
important example of the philosophy of $h$-principles \cite{GrPDR} \cite{EM}.
While embedding seems to be much harder, many relevant questions have likewise
been reduced to algebraic topology at least in principle, with, in our view, the
signal achievements due to Whitney, Haefliger \cite[e.g.]{HaeBki}, and
Goodwillie--Klein--Weiss \cite[e.g.]{GKW}.

Analogous work has been done, to a less complete degree, in the PL category,
with an analogue of the Smale--Hirsch theorem given by Haefliger and Po\'enaru
\cite{HaeP}.

In this paper we discuss, mainly in the case $N=\mathbb{R}^n$ (or, equivalently,
$S^n$), whether these classifications can actually be performed algorithmically
given some finite data representing the pair of manifolds.  This has
consequences not only for computational topology but also for geometry.

In several papers, Gromov emphasized that topological existence results do not
directly enable us to understand the geometric object that is supposed to exist.
Indeed, the eversion of the sphere took quite a while to make explicit (but can
now be observed in several nice animations).  A basic question is: how
complicated are embeddings or immersions, when they exist?

In \cite{CDMW}, an analogous problem was studied in the case of cobordism, which was reduced to homotopy theory in a similar way by the work of Thom.  In that
case, we showed that if a nullcobordism exists, its complexity can be made at worst slightly superlinear with respect to a natural measure of the complexity of the manifold.

In contrast, in the setting of immersion and embedding, there is sometimes no computable upper bound to the complexity of solutions.  Consider smooth $m$-manifolds $M$ with sectional curvature $\lvert K \rvert \leq 1$ and injectivity radius $\geq 1$; we say such manifolds have \emph{bounded local geometry}.  By results of Cheeger and Gromov, these bounds guarantee that the manifold doesn't have ``too much topology'' per unit volume, and volume can therefore be used as a measure of complexity.  In particular, for every $V$, there are finitely many diffeomorphism types of manifolds of bounded local geometry and volume at most $V$.  By taking the maximum over a finite set of manifolds-up-to-small-deformation, we get a function $F_{m,c}(V)$ such that any such smooth manifold $M$, if immersible in $\mathbb{R}^{m+c}$, has an immersion whose bilipschitz constant and norm of
the second derivative are bounded by $F_{m,c}(\vol M)$.

Our undecidability result Theorem \ref{decundec} then implies:
\begin{cor}
  If $c \leq m/4$ and is even, $F_{m,c}(V)$ is not bounded by any computable
  function.
\end{cor}

The proof follows an outline originating in the work of Nabutovsky \cite{NabSph}.  Suppose a computable bound existed.  Then one could solve the decision problem of whether $M$ is immersible by a brute force search \emph{all} candidate functions (up to a $C^2$-small deformation) whose geometry is below the bound.  If an immersion is not found in this search, then the manifold is not immersible.

In this way logical complexity of decision problems is reflected in lower complexity bounds for solutions of related variational problems.  The difference between this case and the situation in \cite{CDMW} is that in the case of cobordism, the relevant algebraic topology is stable homotopy theory, which is algorithmically tractable, while for immersions the relevant problems are unstable.

\subsection{Immersion vs.~embedding}

Some prior work has been done on the decidability of various questions involving
embeddings.  In a pair of papers from the 1990's, Nabutovsky and the second
author \cite{NW} \cite{NW2} considered the problem of recognizing embeddings,
that is, deciding whether two embeddings of a manifold $M$ in a manifold $N$ are
isotopic.  When $M$ and $N$ are both simply connected, this is decidable as long
as the codimension is not 2; in codimension 2, even equivalence of knots
(embeddings of $S^{n-2}$ in $\mathbb{R}^n$) for $n \geq 5$ is not decidable.

A related result asserted in \cite{NW} says that for closed (even
simply-connected) manifolds the problem of embedding is in general undecidable,
as in our paper, for reasons related to Hilbert's tenth problem.  Here, we study
the special case where the target is a sphere, and do not know what to expect
for the case of closed manifolds embedding in the sphere.

More recent work has considered the problem of embedding simplicial complexes in
$\mathbb{R}^n$.  In \cite{MTW} it is shown that this problem is undecidable in
codimensions zero and one, when $n \geq 5$; in \cite{CKV}, that it is decidable
in the so-called \emph{metastable range}, when the dimension of the complex is
at most roughly $\frac{2}{3}n$.

Between codimension $3$ and the metastable range, embedding theory is best
described via the \emph{calculus of embeddings}, due to Goodwillie, Klein and
Weiss.  This describes smooth embeddings of manifolds via a rather complicated
homotopical construction which nevertheless can be arbitrarily closely
approximated via finite descriptions (see e.g.~\cite{GKW}); unlike in the
metastable range, where work of Haefliger shows that immersion theory is
essentially irrelevant, immersions form the ``bottom layer'' of this
construction.  Thus, understanding immersion theory from a computational point
of view seems to be a good first step towards solving this set of problems.  As
we show, it also directly leads to some results regarding embeddings.

While a similar construction for PL embeddings of simplicial complexes is not
currently in the literature, it seems plausible that such a construction can be
developed and will be quite similar to the smooth version.  We believe that many
variations of the embeddability question can eventually be shown to be
undecidable using this correspondence.

\subsection{Summary of old and new results}

The properties of embedding and immersion questions, including their
decidability, depend heavily on the ratio between the dimensions $m$ and $n$ of
the two objects considered.  Our main result concerns the decidability of
immersibility in $\mathbb{R}^n$.
\begin{refthm}{\ref{decundec}}
  The results break up into the following ranges.
  \begin{description}
  \item[The \emph{stable range}, $m<\frac{1}{2}n+1$] The Whitney immersion
    theorem states that every manifold in this range has an immersion in
    $\mathbb{R}^n$.
  \item[The \emph{metastable range}, $\frac{1}{2}n+1 \leq m<\frac{2}{3}n$] For
    manifolds in this range, both smooth and PL immersibility are always
    decidable.
  \item[$\frac{2}{3}n \leq m<\frac{4}{5}n$] In this range, PL immersibility of
    manifolds in $\mathbb{R}^n$ is decidable, as is smooth immersibility as long
    as $n-m$ is odd.  We do not know whether smooth immersibility in even
    codimension is decidable.
  \item[$\frac{4}{5}n \leq m \leq n-3$] In this range, PL immersibility of
    manifolds is decidable, whereas smooth immersibility is decidable if and
    only if $n-m$ is odd.
  \item[$m=n-2$] In codimension 2, there are two notions of PL
    immersion: in a \emph{locally flat} immersion, links of vertices are always
    unknotted in the ambient space; but one may also study PL immersions which
    are not necessarily locally flat.  Here, smooth immersibility is undecidable
    at least when $n \geq 10$, as is PL locally flat immersibility, which is
    equivalent; PL not necessarily locally flat immersibility is decidable.
  \item[$m=n-1$] In codimension 1, immersibility is decidable.
  \end{description}
\end{refthm}
This parallels the overall picture for embedding theory, about which we still
know much less.  Note that the stable and metastable ranges are slightly
different from the immersion case.
\begin{description}
\item[The stable range, $m \leq n/2$] The Whitney embedding theorem states
  that every manifold in this range has an embedding in $\mathbb{R}^n$.  For
  simplicial complexes, one needs $m<n/2$; for $n=2m$, embeddability is
  obstructed by the Van Kampen obstruction.
\item[The metastable range, $m \leq \frac{2}{3}n-1$] Here the
  embeddability of simplicial complexes is decidable; this is a theorem of
  \v Cadek, Kr\v c\'al and Vok\v r\'inek \cite{CKV}.  Moreover, PL embeddings in
  this range are smoothable, so smooth embeddability is decidable as well.
\item[$\frac{2}{3}n \leq m \leq \frac{10}{11}n$] In this range, nothing is known
  about whether embeddability is decidable; however, see \cite{MTW} for some
  lower bounds on computational complexity.  Moreover, ongoing work of
  Filakovsk\'y, Wagner and Zhechev on the embedding extension problem (is it
  possible to extend an embedding of a subcomplex to an embedding of the whole
  space?) suggests that the more general problem of \emph{classifying}
  embeddings of simplicial complexes up to isotopy is undecidable in the vast
  majority of this range.
\item[$\frac{10}{11}n<m \leq n-2$] The state of the art on embeddability of
  simplicial complexes is much the same here as in the previous range.  However,
  our results on immersions are enough to show:
  \begin{refthm}{\ref{thm:emb}}
    When $\frac{10}{11}n<m \leq n-2$ and $n-m$ is even, the embeddability of a
    smooth $m$-manifold with boundary in $\mathbb{R}^n$ is undecidable.
  \end{refthm}
  \noindent However, the reduction (from Hilbert's tenth problem via the
  immersion problem) which we use to show undecidability creates examples that
  are always PL embeddable; the construction relies on the smooth structure of
  the manifold.  Moreover, we do not know whether embeddability is decidable
  when restricted to closed manifolds; as discussed in \S\ref{S:closed}, the
  method of Theorem \ref{thm:emb} cannot work in that case.
\item[$n=m-1$] Here, PL embeddability is undecidable, as shown in \cite{MTW}.
\end{description}

\subsection{Methods}
Questions of immersibility and embeddability are classically handled by reducing
them first to pure homotopy theory and then reducing the homotopy theory to
algebra.  To resolve any particular instance, then, one has to do the
corresponding algebraic computation.  To decide whether the answers can be
obtained algorithmically, one has to (1) find an algorithm to perform the
reduction and (2) determine whether the resulting algebra problem is decidable.

The homotopy-theoretic side of these questions is fairly well-studied.  Novikov
showed in 1955 that it is undecidable whether a given finite presentation yields
the trivial group; in particular, this means that whether a given simplicial
complex is simply connected is undecidable.  This was extended by Adian to show
that many other properties of groups are likewise undecidable.  Soon after,
Brown \cite{Brown} showed, by way of contrast, that the higher homotopy groups
of a simply connected space are computable.

Much more recently, \v Cadek, Kr\v c\'al, Matou\v sek, Vok\v r\'inek and Wagner
\cite{CKMVW} showed that the set of homotopy classes $[X,Y]$ is in general
uncomputable, even when $Y$ is a simply connected space.  This is because the
problem of determining which rational invariants can be attained is tantamount
to resolving a system of diophantine equations; this is the famously undecidable
Hilbert's tenth problem.

It seems as if fundamental group issues and Hilbert's tenth problem are the only
obstructions to computability in homotopy theory.  The same group of authors,
along with Filakovsk\'y, Franek, and Zhechev, have authored a number of papers
\cite{CKMVW2,FiVo,Vok,CKV,FFWZ} describing algorithms for various problems in
homotopy theory that do not encounter these.  While some open questions do
remain, all of the homotopy-theoretic problems encountered in this paper can
easily be reduced to ones covered by their results.

The main issue, then, is that of the reduction.  The $h$-principles of
Hirsch--Smale \cite{Hirsch} and Haefliger--Po\'enaru \cite{HaeP}, respectively,
show that immersions of codimension $k$ in the smooth and PL categories are
classified via lifts of the stable tangent bundle to the classifying spaces
$BO_k$ and $BPL_k$.  While $BO_k$ can easily be approximated by a Grassmannian
of $k$-planes in a high-dimensional Euclidean space, and therefore classifying
maps are also not difficult to compute, $BPL_k$ is more recalcitrant.  While it
is known to be of finite type, that is, homotopy equivalent to a complex with
finite skeleta, this equivalence is inexplicit and it is not clear how to
algorithmically reduce the tangential data of a PL manifold to a finite amount
of data.  In this paper, we employ various workarounds; the question of
understanding $BPL$ more directly remains open and is also relevant to the
quantitative topology of PL manifolds.

\subsection{Complexity}

Our algorithms do not give any information about the complexity of the
computations.  In many cases, we perform a construction by iterating through all
objects of a given form until we find the needed one; this uses the fact that
its existence is known and that it is algorithmically recognizable.  However,
often such an object only exists when the input is a manifold; this means that
the algorithm will not terminate if presented with an invalid input (for
example, a simplicial complex all of whose links are homology spheres, but not
spheres.)

We believe that this issue can be circumvented and that these algorithms can be
made much more efficient, but this is beyond the scope of this paper.

\subsection{Acknowledgements}

The authors would like to thank the Israel Institute for Advanced Studies for a
stay in fall 2017 during which this work was conceived; Uli Wagner for his
encouragement and helpful comments; and several referees for keeping us honest by
calling out a great deal of handwaving in successive drafts, as well as
suggesting streamlined arguments in a few places.

\section{Effective representation}

In this section, we discuss algorithms and data structures which we borrow from
previous work, as well as new ones to represent certain objects which have not
been worked out in detail previously.

\subsection{Computational homotopy theory}
There has been a fair amount of work on computational homotopy theory: taking a
finite simplicial complex and algorithmically describing its homotopy groups,
Postnikov tower, and so forth.  Here we summarize some of this past work and
cite algorithms for various operations which we will use as building blocks.
\begin{prop} \label{off-the-shelf}
  \begin{enumerate}[(a)]
  \item The isomorphism type of homotopy groups of spheres is computable
    \cite{Brown}.
  \item Moreover, there is an algorithm that, given a finite simply connected
    simplicial complex $X$, computes a generating set and relations for
    $\pi_k(X)$ and an explicit simplicial representative for each generator
    \cite{FFWZ}.  From this, one can compute a simplicial representative for any
    linear combination of the generators.
  \item Given a map $Y \to B$ between simply connected finite simplicial
    complexes or simplicial sets, its relative (or Moore--)Postnikov tower may
    be computed to any finite stage.  The output is given in a form so that the
    (co)homology of every Postnikov stage $P_n$ may be computed, as well as the
    maps
    \[Y \to P_n \to P_{n-1} \to B\]
    and the maps on (co)homology induced by them \cite[Theorem 3.3]{CKV}.
  \item \label{lifts} Given a diagram
    \[\xymatrix{
      A \ar[r] \ar@{^(->}[d] & P_n \ar@{->>}[d] \\
      X \ar[r] \ar@{-->}[ru] & P_{n-1},
    }\]
    where $(X,A)$ is a finite simplicial pair and $P_n \to P_{n-1}$ is a
    (relative) Postnikov stage, the obstruction in $H^n(X,A;\pi)$ to filling in
    the dotted arrow can be computed.  Moreover, if the obstruction is zero, a
    lifting-extension may be constructed \cite[Proposition 3.7]{CKV}.  Finally,
    given a lifting-extension $f:X \to P_n$ and a cochain
    $w \in C^{n-1}(X,A;\pi)$, where $P_n \to P_{n-1}$ is a
    $K(\pi,n-1)$-fibration, one can construct another lifting-extension
    $g:X \to P_n$ such that the obstruction to homotoping $f$ and $g$ is
    $[w] \in H^{n-1}(X,A;\pi)$.
  \item \label{homotopic} Given two maps $X \to S^n$, for any simplicial complex
    $X$, there is an algorithm to determine whether they are homotopic; more
    generally, one can replace $S^n$ with any simply connected finite complex
    $Y$ \cite{FiVo}.  In particular, given an explicit map $S^k \to S^n$, one
    can (by iterating over the combinations) determine its homotopy class as a
    combination of the generators computed in \cite{FFWZ}.
  \item \label{nullh} Given a map $X \to Y$ known to be nullhomotopic, we can
    compute an explicit nullhomotopy.
  \end{enumerate}
\end{prop}
\begin{proof}
  We prove only the parts which are not given a citation in the statement.

  \textbf{Part \eqref{lifts}.} The last part is not explicitly done in
  \cite{CKV}, but the representation of Postnikov stages given there makes it
  easy to modify a map $f:X \to P_n$ by a cochain in $C^{n-1}(X;\pi)$.

  \textbf{Part \eqref{nullh}.} This can be done through an exhaustive search for
  maps from increasingly fine subdivisions of the cone on $X$.
\end{proof}

\subsection{Smooth manifolds}
There are several possible ways of representing smooth manifolds
computationally; as far as we know, this topic has not been thoroughly explored.
According to the Nash--Tognoli theorem, every smooth $n$-manifold embedded in
$\mathbb{R}^{2n+1}$ is closely approximated by a smooth real algebraic variety
cut out by rational polynomials.  This is one way of specifying smooth
manifolds, however it is not clear whether it can be computed from other
possible representations.

\subsubsection{Our model} \label{S:smooth}
For our purposes, compact smooth $n$-manifolds will be specified via $C^1$
triangulations.  (Classically, the categories of $C^1$ and $C^\infty$ manifolds are equivalent, and $C^1$ immersions or embeddings can be approximated by $C^\infty$ ones \cite[\S4]{Munkres}.  Therefore we treat ``smooth'' and ``$C^1$'' as synonymous.)  That is, we take a simplicial complex and specify a polynomial map with rational coefficients of each top-dimensional simplex to $\mathbb{R}^N$, for some $N$, such that the derivatives are nonsingular and
coincide where simplices meet.  (To be precise, given adjacent simplices
$\delta_1,\delta_2:\Delta^n \to \mathbb R^N$, and viewing $\Delta^n$ as a subset
of $\mathbb R^n$, there is an affine transformation
$L:\mathbb R^n \to \mathbb R^n$ such that $\delta_1$ and the map
$\delta_2 \circ L:L^{-1}(\Delta^n) \to \mathbb R^N$ patch together to form a
$C^1$ immersion of the connected set $\Delta^n \cup L^{-1}(\Delta^n)$ in $\mathbb R^N$.)
\begin{rmk}
  To represent every diffeomorphism type of smooth $n$-manifold, it suffices to
  use polynomials of bounded degree, with the bound depending on $n$.  This is
  because, for any smoothly embedded manifold, there is a piecewise polynomial
  approximation of the embedding using splines on a triangulation: start by
  fixing the tangent spaces at vertices, then interpolate via cubic functions on
  each edge, and so on.  If the triangulation is sufficiently fine, so that the
  embedding is close to linear on each simplex, then the approximation will also
  be an embedding.
\end{rmk}

Now suppose that we have a simplicial complex $M$ and map $f:M \to \mathbb R^N$
satisfying these properties.  Under what conditions does this actually define a
closed manifold smoothly immersed in $\mathbb R^N$?  Since there is a
well-defined tangent space at every point, it suffices to show that for every
vertex $v$, the ``derivative'' $Df_v:\st(v) \to T_vM$ (that is, the map whose
restriction to each simplex of the star is the derivative at $v$ of the
restriction of $f$ to that simplex) is a bijection to a neighborhood of the
origin; then these derivatives at every vertex give an atlas of $C^1$ charts for
$M$.  The derivative is easily computed from the polynomial maps defining $f$.
To show that it is injective, it suffices to show that the intersection between
the images of any two simplices is the image of their intersection, which is a
matter of showing that a set of linear inequalities is unsatisfiable.  To show
that it is surjective, it then suffices to show that $H_{n-1}(\lk(v))$ is
nontrivial, since any proper subset of $S^{n-1}$ has trivial $(n-1)$st homology.
Thus whether a given piece of data represents an immersed manifold can be
decided algorithmically.

Since every $n$-manifold embeds in $\mathbb R^{2n+1}$ and every embedding can be
approximated by a $C^1$ piecewise polynomial map with rational coefficients on
a sufficiently fine triangulation, this gives a way of enumerating all smooth
closed $n$-manifolds: one iterates through all pure $n$-dimensional simplicial
complexes and piecewise polynomial maps and checks the conditions to decide
whether the data defines a closed manifold.

We can enumerate compact smooth manifolds with boundary using a similar
strategy.  To check whether a set of data encodes a smooth $n$-manifold $M$ with
boundary, we need to check whether the boundary (defined combinatorially) is a
smooth $(n-1)$-manifold; whether $M$ smoothly embeds at interior vertices; and,
for each boundary vertex $v$, that $Df_v$ injects into the tangent half-space
and that the relative homology group $H_{n-1}(\lk_M(v),\lk_{\partial M}(v))$ is
nontrivial.

\subsubsection{Grassmannians and Pontryagin classes} \label{S:Grass}
Determining explicit triangulations of real Grassmannians is an interesting and
apparently open problem in combinatorial algebraic geometry.  However, it is
well-established that some triangulation can be computed algorithmically.  The
oriented Grassmannian $\Gr_n(\mathbb R^N)$ can be explicitly expressed as an
algebraic variety, for example in $\Lambda^n \mathbb R^N$ via the (double cover
of the) Pl\"ucker embedding which sends an oriented $n$-plane to
$v_1 \wedge \ldots \wedge v_n$, for any ordered orthonormal basis
$v_1,\ldots,v_n$.  And in fact a triangulation can be computed algorithmically
for any semi-algebraic set \cite[Remark 11.19(b)]{BPR}.

The Pl\"ucker embedding also has the advantage that, given the derivative of a
map from an $n$-manifold $M$ to $\mathbb R^N$, the corresponding point in the
Grassmannian can be readily computed as the normalized wedge product of its
columns.  In particular we can compute the value of the classifying map
$\ph:M \to \Gr_n(\mathbb R^N)$ of the tangent bundle at any rational point of
any simplex.  Likewise:
\begin{lem}
  One can compute an upper bound for the Lipschitz constant of $\ph$ with
  respect to the induced Riemannian metrics on $M$ and $\Gr_n(\mathbb R^N)$.
\end{lem}
\begin{proof}
  The global Lipschitz constant is the maximum of the Lipschitz constants of
  each polynomial piece.  On each polynomial piece, an upper bound for $\Lip\ph$
  is given by
  \[\frac{\max\{\lVert D\Phi_x(v) \rVert : x \in \Delta^n, v \in \mathbb R^n, \lVert Df_x(v) \rVert=1\}}{\min \{\lVert \Phi(x) \rVert : x \in \Delta^n\}},\]
  where $\Phi(x)=(Df_x)_1 \wedge \cdots \wedge (Df_x)_n$ is the pre-normalization
  version of $\ph$.  The numerator and denominator are instances of optimization
  of a polynomial function over a semialgebraic set, which is computable via
  \cite[Algorithm 14.9]{BPR}.
\end{proof}

\begin{lem} \label{simp-approx}
  This data suffices to produce a simplicial approximation of $\ph$.
\end{lem}
\begin{proof}
  Given a triangulation of the Grassmannian, thought of as a map
  $\tau:K \to \Gr_n(\mathbb R^N)$ where $K$ is a simplicial complex equipped
  with the standard simplexwise linear length metric, let
  \[s=\frac{\text{inradius of a standard $\dim(\Gr_n(\mathbb R^N))$-simplex}}{\Lip(\tau^{-1})\Lip \ph}.\]
  Note that the inradius of an $r$-simplex is given by
  $\frac{1}{\sqrt{2r(r+1)}}$.  Moreover, $\Lip(\tau^{-1})$ can be computed
  algorithmically.  For every simplex, one must minimize the derivative of a
  polynomial function over that simplex, which can again be done via
  \cite[Algorithm 14.9]{BPR}.  Therefore we can compute a lower bound for $s$.

  The constant $s$ is defined so that for any ball $B$ of radius $s$ in $M$,
  $\ph(B)$ is contained in the star of some vertex of $\tau$.  Suppose we have
  a subdivision $\tau''$ of our triangulation $\tau'$ of $M$ in which the
  diameter of every simplex is at most $s/2$.  By the usual proof of the
  simplicial approximation theorem, see e.g.~\cite[Theorem 2C.1]{Hatc}, the map
  from the $0$-skeleton of $\tau''$ taking each vertex $v$ to the vertex of
  $\tau$ nearest to $\ph(v)$ extends linearly to a simplicial map homotopic to
  $\ph$.

  To compute such a subdivision, we first compute the Lipschitz constant $C$ of
  the triangulation $\tau'$ of $M$ (this boils down to maximizing derivatives of
  polynomial functions with linear constraints).  Then we compute the
  $k$-fold edgewise subdivision of $\tau'$ in the sense of \cite{EdGr}, for some
  $k>4C/s$ (that is, each edge of $\tau'$ is divided into $k$ edges).  This
  subdivision has rational vertices and the simplices of the $k$-fold edgewise
  subdivision of the standard simplex have diameter at most $2/k$.  Therefore
  the diameters of our new simplices are at most $s/2$.
\end{proof}

We would like to use this to compute the rational Pontryagin classes of $M$, by
pulling back simplicial cochains on $\Gr_n(\mathbb R^N)$.  Since
$\Gr_n(\mathbb R^N)$ is a finite simplicial complex, we can compute its rational
cohomology algorithmically as an abelian group finitely generated by explicit cochains, whose simplicial cup products we can also take.

To determine the generators corresponding to the Pontryagin classes $p_1,\ldots,p_{n/4}$, we can use the Pontryagin numbers of products of complex projective spaces.  If $n$ is a multiple of $4$, for every partition $n/4=i_1+\cdots+i_r$, the Pontryagin numbers of $\mathbb CP^{2i_1} \times \cdots \times \mathbb CP^{2i_r}$ are known explicitly, and by a result of Thom \cite[Theorem 16.8]{MilSta}, the matrix of Pontryagin numbers of all such $n$-dimensional products is nonsingular.  If (by induction) we have computed cochain representatives of $p_1,\ldots,p_{n/4-1}$ in $H^*(\Gr_n(\mathbb R^N); \mathbb Q)$, and we have constructed classifying maps for the various $\mathbb CP^{2i_1} \times \cdots \times \mathbb CP^{2i_r}$, then we can solve for $p_{n/4} \in H^n(\Gr_n(\mathbb R^N);\mathbb Q)$.

We can compute explicit embeddings of such products in $\mathbb R^N$ for $N$
sufficiently large, using an explicit embedding of $\mathbb CP^{2i_j}$ as an
affine algebraic variety.  Since Pontryagin classes are stable, we can therefore
compute the corresponding cohomology classes in $\Gr_{n+k}(\mathbb R^{N+k}; \mathbb Q)$ for any $k$.  If we need to reduce $N$, we can pull the classes back along a simplicial approximation to the inclusion
$\Gr_n(\mathbb R^{N_0}) \to \Gr_n(\mathbb R^N)$.  This proves:
\begin{prop}
  There is an algorithm to construct a triangulation of $\Gr_n(\mathbb R^N)$ and
  compute cochain representatives of the Pontryagin classes
  $p_1,\ldots,p_{\lfloor n/4 \rfloor} \in H^*(\Gr_n(\mathbb R^N);\mathbb Q)$.
\end{prop}
Our algorithms will not require the computation of
$p_{\lfloor n/4 \rfloor+1},\ldots,p_{\lfloor n/2 \rfloor}$.

\subsubsection{Other possible models}
We note that there are a number of other ways of specifying a smooth manifold
via a combinatorial structure.  We list some of these here; the extent to which
they can be transformed into each other requires further study.

A more general way of specifying smooth (that is, $C^1$) $n$-dimensional
submanifolds of $\mathbb{R}^N$ is by patching them together from smooth real
semialgebraic sets, with a consistent derivative along the boundaries of the
patches.  This includes the case of a single smooth variety; a triangulated
smooth manifold with semialgebraic simplices; and a handle decomposition with
semialgebraic handles.

One can dispense with the explicit embedding by taking a triangulated manifold
and assigning an element of $\Gr_n(\mathbb{R}^N)$ to each vertex.  If the
triangulation is sufficiently fine, we can send adjacent vertices close enough
to each other (at most some constant distance depending on $n$ and $N$) that one
can interpolate linearly over the simplices, uniquely determining a smooth
structure on the manifold.  One must ensure, of course, that this structure is
compatible with the PL structure.

Finally, one can specify a manifold via an atlas of coordinate patches and
transition functions; for example one may require the patches to be real
semialgebraic and the transition functions to be rational functions (in one
direction).

\subsection{Classifying spaces for spherical fibrations} \label{S:BGn}
Manifold topology makes use of a variety of classifying spaces, the most
familiar of which are the classifying spaces $BO_n$ and $BSO_n$ for unoriented
and oriented vector bundles.  These have relatively straightforward models as
Grassmannians of $n$-planes in $\mathbb{R}^\infty$.  The classifying spaces
$BPL_n$ for $PL$ structures are much more complicated to model combinatorially;
while some work in this direction has been done by Mn\"ev \cite{Mnev}, in this
paper we do not attempt to make $BPL_n$ and classifying maps to it concrete
enough to manipulate algorithmically.  Instead, we use some well-known
computations to avoid talking about $BPL_n$ at all and focus instead on $BG_n$,
the classifying space for the much weaker structure of $S^{n-1}$-fibrations.

Transition functions between fibers in an oriented $S^{n-1}$-fibration are chosen
from the topological monoid $G_n$ of homotopy automorphisms of $S^{n-1}$
homotopic to the identity; this fits into a fiber sequence
\[\Omega^{n-1}_{\pm \id}S^{n-1} \to G_n \to S^{n-1},\]
where the fiber consists of degree $\pm 1$ maps $S^{n-1} \to S^{n-1}$ in the
iterated loop space $\Omega^{n-1}S^{n-1}$ and the fibration is induced by
evaluation at the basepoint.  This monoid has a classifying space $BG_n$; as
$n \to \infty$, this converges to a stable object $BG$.  In order to compute
with $BG_n$ and $BG$, we need to construct finite models for their skeleta as
well as the tautological bundles over them.
\begin{lem} \label{lem:BGn}
  \begin{enumerate}[(i)]
  \item There is an algorithm which, given natural numbers $m$ and $n$,
    constructs a finite simplicial set $B_{m,n}$ which has an $m$-connected map
    to $BG_n$, together with any stage of the relative Postnikov tower of the
    pullback to $B_{m,n}$ of the tautological bundle over $BG_n$.
  \item There is an algorithm that constructs the map $B_{m,n} \to B_{m,m+1}$
    induced by iterated suspension.  Since the map $BG_{m+1} \to BG$ is
    $m$-connected, this map models the stabilization $BG_n \to BG$.
  \item \label{BGniii} There is an algorithm which, given a stable range PL
    embedding $M^m \to \mathbb{R}^{2m+k}$, $k \geq 1$, constructs the classifying
    map $M \to B_{m,m+k}$ of the normal bundle.
  \end{enumerate}
\end{lem}
We note that the model we compute is extremely inexplicit in how it classifies
fibrations.
\begin{proof}
  We start by outlining the algorithm for (i).  We will use the algorithms
  outlined in Proposition \ref{off-the-shelf} mostly without comment as building
  blocks.

  We first show that we can compute $\pi_k(G_n)$, given $n$ and $k$.  From the
  fiber sequence
  \[\Omega^{n-1}_{\pm \id}S^{n-1} \xrightarrow{i} G_n \xrightarrow{j} S^{n-1},\]
  we obtain the homotopy exact sequence
  \[\pi_{k+1}(S^{n-1}) \xrightarrow{\phi_{k+1}} \pi_{n+k-1}(S^{n-1})
  \xrightarrow{i_*} \pi_k(G_n) \xrightarrow{j_*} \pi_k(S^{n-1})
  \xrightarrow{\phi_k} \pi_{n+k-2}(S^{n-1}).\]
  Thus it is enough to perform the following steps:
  \begin{enumerate}
  \item Compute representatives of generators for $\pi_k(S^{n-1})$ and
    $\pi_{k+1}(S^{n-1})$.
  \item Compute the obstruction theoretic map
    $\phi_k:\pi_k(S^{n-1}) \to \pi_{n+k-2}(S^{n-1})$.  Given a map
    $f:S^k \to S^{n-1}$, $\phi_k(f)$ is the obstruction to lifting $f$ to a map
    $\tilde f:S^k \to \Omega_{\id}^{n-1}S^{n-1}$ which sends the base point to the
    identity map.  In other words, it is the obstruction to extending the map
    \[f \vee \id:S^k \vee S^{n-1} \to S^{n-1}\]
    to $S^k \times S^{n-1}$.  This obstruction is the \emph{Whitehead product}
    $[f,\id_{S^{n-1}}]$: the composition
    \[S^{n+k-2} \to S^k \vee S^{n-1} \xrightarrow{f \vee \id} S^{n-1}\]
    where the first map is homotopic to the attaching map of the top cell of
    $S^k \times S^{n-1}$.  From this map, we compute its homotopy class as an
    element of $\pi_{n+k-2}(S^{n-1})$; doing this for a representative of each
    generator gives a finite description of the map $\phi_k$.
  \item Now $\pi_k(G_n)$ is generated by lifts of $\ker\phi_k$ and the image of
    $\pi_{n+k-1}(S^{n-1})$ under $i_*$.  We compute, via exhaustive search, a
    homotopy lift of each generator of $\ker\phi_k$ to a map
    $S^k \times S^{n-1} \to S^{n-1}$; the generators of $\pi_{n+k-1}(S^{n-1})$
    determine maps $S^k \times S^{n-1} \to S^{n-1}$ by precomposing with the map
    collapsing $S^k \vee S^{n-1}$.  We then determine the isomorphism type of
    $\pi_k(G_n)$ by computing all relations.

    When both groups are finite, this is a finite computation.  This leaves the
    following cases:
    \begin{description}
    \item[$k=n-2$, $n$ odd] Then
      \[\pi_{n+k-1}(S^{n-1})=\pi_{2n-3}(S^{n-1}) \cong \mathbb{Z} \oplus A_{n-1},\]
      where $A_{n-1}$ is a finite group, and the restriction of the map
      \[\phi_{k+1}:\pi_{k+1}(S^{n-1}) \to \pi_{n+k-1}(S^{n-1})\]
      to the $\mathbb Z$ factor is either surjective or has cokernel
      $\mathbb{Z}/2\mathbb{Z}$, depending on the resolution of the Hopf
      invariant one problem in that dimension, since $[\id,\id]$ always has Hopf
      invariant two.  Thus $\img i_*$ is finite; this can be hardcoded into the
      computation.
    \item[$k=n-1$, $n$ even] In this case $\ker\phi_k \cong \mathbb{Z}$, whereas
      $\pi_{n+k-1}(S^{n-1})$ is finite.  Therefore
      \[\pi_k(G_n) \cong \mathbb{Z} \oplus \img(i_*),\]
      and we only need to compute the relations within the image of $i_*$.
    \item[$k=2n-3$, $n$ odd] In this case $\ker\phi_k$ contains a $\mathbb{Z}$
      factor and a finite factor $A$, and $\pi_{k+n-1}(S^{n-1})$ is again finite.
      This makes this case similar to the previous one:
      \[\pi_k(G_n) \cong \mathbb{Z} \oplus B,\]
      where $B$ is an extension of $A$ by $\pi_{k+n-1}$, a finite group all of
      whose relations can be computed.
    \end{description}
    In addition, the case $k=0$ must be coded separately: $G_n$ always has two
    components.
  \end{enumerate}

  Now we use the fact that $\pi_k(G_n)=\pi_{k+1}(BG_n)$ to build successive
  approximations $B_{i,n}$ of $BG_n$, together with the pullbacks
  $p_i:E_{i,n} \to B_{i,n}$ of the tautological $S^{n-1}$-fibration.

  Given a map $f:S^k \times S^{n-1} \to S^{n-1}$ representing an element
  $\alpha \in \pi_k(G_n)$ (in particular with $f|_{* \times S^{n-1}}=\id$) we define
  the space
  \[E_f=D^{k+1} \times S^{n-1}/\{(x,y) \sim (*,f(x,y)) : x \in \partial D^{k+1},
  y \in S^{n-1}\};\]
  then the projection $E_f \to S^{k+1}$ onto the first factor has homotopy fiber
  $S^{n-1}$, and $\alpha$ is the obstruction to constructing a fiberwise homotopy
  equivalence $S^{n-1} \times S^{k+1} \to E_f$.  We use the $E_f$ for
  representatives of a generating set for $\pi_k(G^n)$ as building blocks for
  our construction.

  We set $p_1:E_{1,n} \to B_{1,n}$ to be the map
  \[\bigvee_{[f]\text{ generating }\pi_0(G_n)} E_f \to \bigvee S^1.\]
  Now suppose we have constructed $p_i:E_{i,n} \to B_{i,n}$ which is the homotopy
  pullback of the tautological bundle over $BG_n$ along an $i$-connected map.
  Then we construct the spaces $E_{i+1,n}$ and $B_{i+1,n}$ and the map $p_{i+1}$
  using the following algorithm.  Here the CW structure can be given via
  simplicial maps from subdivided simplices corresponding to each cell.
  \begin{enumerate}
  \item First, we compute the kernel of the map $\pi_i(B_{i,n}) \to \pi_i(BG_n)$;
    since $E_{i,n}$ is the pullback of the tautological bundle, this means
    determining which elements of $\pi_i(B_{i,n})$ pull $E_{i,n}$ back to a
    trivial fibration over $S^i$.

    To do this, we first evaluate, for each generator of $\pi_i(B_{i,n})$, the
    obstruction in $\pi_{i-1}(S^{n-1})$ to lifting it to $E_{i,n}$.  This allows us
    to compute the kernel $K$ of this obstruction, as well as a generating set
    for this kernel and explicit lifts $h_j:S^{n-1} \to E_{i,n}$ of the generators
    of $K$.

    The obstruction to extending $h_j$ to a fiberwise homotopy equivalence
    $S^i \times S^{n-1} \to E_{i,n}$ depends on the choice of lift; as above, it
    is the Whitehead product $[h_j,\iota]$ where $\iota:S^{n-1} \to E_{i,n}$ is
    the inclusion of a fiber.  For each $h_j$
    and each $g:S^i \to S^{n-1}$ representing a generator of $\pi_i(S^{n-1})$, we
    compute the obstruction in $\pi_{n+i-2}(S^{n-1})$ to completing the diagram
    \[\xymatrixcolsep{4pc}\xymatrix{
      S^i \vee S^{n-1} \ar[r]^-{(h_j+\iota g) \vee \iota} \ar@{^(->}[d] & E_{i,n} \ar@{->>}[d]^{p_i} \\
      S^i \times S^{n-1} \ar[r]^-{\overline{h_j}} \ar@{-->}[ru] & B_{i,n},
    }\]
    where $\overline{h_j}(x,y)=p_i \circ h_j(y)$.  These obstructions generate a
    unique homomorphism
    \[K \times \pi_i(S^{n-1}) \to \pi_{n+i-2}(S^{n-1})\]
    which describes the obstruction to extending any lift of any element of $K$.
    This allows us to compute the kernel of this obstruction as a subgroup of
    $K \times \pi_i(S^{n-1})$.  Its projection to $K$ is the subgroup of
    $\pi_i(B_{i,n})$ we desire.

    Note that both of these obstruction-theoretic calculations only depend on
    one stage of the Postnikov tower of $E_{i,n} \to B_{i,n}$ and therefore fit
    into the computational framework of \cite{CKV} described in Proposition
    \ref{off-the-shelf}\eqref{lifts}.
  \item Now given a generating set for this kernel, we glue in an $(i+1)$-cell
    for each generator to $B_{i,n}$ and a corresponding copy of
    $D^{i+1} \times S^{n-1}$ to $E_{i,n}$.  This ensures that the map
    $\pi_i(B_{i+1,n}) \to \pi_i(BG_n)$ is an isomorphism.
  \item Finally, we wedge on $E_f \to S^{i+1}$ for a set of functions $f$ which
    generate $\pi_i(G_n)$.  This ensures that $B_{i+1,n} \to BG_n$ is an
    $(i+1)$-connected map.
  \end{enumerate}
  Finally, we construct the relative Postnikov tower of the map
  $p_m:E_{m,n} \to B_{m,n}$.  This completes the proof of (i).

  For both (ii) and (iii), we will need a subroutine which, given a map
  $f:E \to B$ whose homotopy fiber is $S^{n-1}$ and such that $B$ is
  $m$-dimensional, computes the classifying map to $B_{m,m+1}$.  We note that
  since the homotopy groups of $BG$ are finite, so are the homotopy groups of
  $B_{m,m+1}$ through dimension $m$.  Therefore there are a finite number of
  homotopy classes of maps $B \to B_{m,m+1}$, which can be enumerated via
  obstruction theory; we choose the one for which $f$ is the homotopy pullback
  of $p_m$, which can be verified by induction on the relative Postnikov tower.

  For (ii), we can construct the map $f$ by repeatedly taking the double mapping
  cone of $p_m$.

  For (iii), given a PL embedding of $M$, we need to compute the Spivak normal
  $S^{m+k-1}$-fibration $E \to M$ before we compute its classifying map.  It is
  enough to find the following data:
  \begin{itemize}
  \item a compact PL $(2m+k)$-manifold with boundary $N(M)$ embedded in
    $\mathbb{R}^{2m+k}$ which contains a subdivision of $M$ in its interior;
  \item a strong deformation retraction of $N(M)$ to $M$.
  \end{itemize}
  Then the induced map $\partial N \to M$ is the Spivak normal fibration.  Since
  these properties are checkable and we can iterate through all subdivisions of
  $M$, simplicial complexes in $\mathbb{R}^{2m+k}$ with rational vertices, and
  simplicial maps from a subdivision of $N(M) \times I$ to $M$, we can find this
  data via exhaustive search.
\end{proof}

\section{Immersibility}

\begin{thm} \label{decundec}
  Let $n \geq 4$ and $m<n$ be natural numbers.
  \begin{enumerate}[(i)]
  \item Whenever $n-m$ is odd or $3m \leq 2n-1$, the immersibility of a smooth
    $m$-manifold with boundary (given as a semialgebraic set in some
    $\mathbb{R}^N$) in $\mathbb{R}^n$ is algorithmically decidable.
  \item \label{undec}
    Whenever $n-m$ is even and $5m \geq 4n$, the immersibility of a smooth
    $m$-manifold in $\mathbb{R}^n$ is undecidable (including if only closed
    manifolds are considered.)
  \item Whenever $n-m \neq 2$, the immersibility of a PL $m$-manifold with
    boundary in $\mathbb{R}^n$ is decidable.
  \item When $n-m=2$, it is undecidable (at least for $n \geq 10$) whether a PL
    $m$-manifold has a locally flat immersion in $\mathbb{R}^n$, but there is an
    algorithm to decide whether it has a not necessarily locally flat immersion.
  \end{enumerate}
  Moreover, over the cases for which an algorithm exists, it can be made uniform
  with respect to $m$ and $n$.
\end{thm}
Note that for certain pairs with $n-m$ even, we have not determined whether
immersibility is decidable.  We suspect that it is in fact undecidable in those
cases, since the corresponding homotopy-theoretic problem is undecidable.
\begin{proof}
  We assume at first that the manifold is oriented, to avoid fundamental group
  issues.

  In each case, the problem of immersibility can be reduced to a homotopy
  lifting problem: these are the $h$-principles of Smale--Hirsch \cite{Hirsch}
  in the smooth case and Haefliger--Poenaru \cite{HaeP} in the PL case.  Both of
  these results state that the space of immersions $M \to N$ in the appropriate
  category is homotopy equivalent to the space of tangent bundle monomorphisms
  $TM \to TN$, or simply $TM \to \mathbb{R}^n$ when $N=\mathbb{R}^n$.  These in
  turn can be thought of as lifts of the classifying map of the tangent bundle
  to the Grassmannian of $m$-planes in $\mathbb{R}^n$.

  \subsubsection*{The smooth case}
  The previous paragraph reduces immersibility to the homotopy lifting problem
  \[\xymatrix{
    & \Gr_m(\mathbb{R}^n) \ar[d] \\
    M \ar[r]^-\kappa \ar@{-->}[ru] & BSO_m
  }\]
  where $\kappa$ is the classifying map of the tangent bundle of $M$.  Moreover,
  this lifting property is \emph{stable}: such a lift exists if and only if the
  corresponding lift
  \[\xymatrix{
    & \Gr_{m}(\mathbb{R}^n) \ar[r] \ar[d] & \Gr_{m+1}(\mathbb{R}^{n+1}) \ar[d] \\
    M \ar[r]_-\kappa \ar@{-->}[rru] & BSO_m \ar[r] & BSO_{m+1} 
  }\]
  exists.  This is because the map between the corresponding homotopy fibers
  \[V_m(\mathbb{R}^n) \to V_{m+1}(\mathbb{R}^{n+1})\]
  is $(n-1)$-connected.  Therefore, when $m<n$, it suffices to resolve the
  lifting problem
  \begin{equation} \label{eqn:BSO}
    \begin{gathered}
    \xymatrix{
    & BSO_{n-m} \ar[d] \\
    M \ar[r]^-\kappa \ar@{-->}[ru] & BSO
    }
    \end{gathered}
  \end{equation}
  which arises from the limit of this sequence.

  The spaces $BSO_{n-m}$ and $BSO$ each have the rational homotopy type of a
  product of Eilenberg--MacLane spaces \cite[Prop.~15.15]{FHT}.  In particular,
  the rational homotopy groups are dual to a subset of the rational cohomology
  algebra.  Moreover, a map to either of these spaces is determined up to a
  finite set by the pullbacks of cohomology generators.  Specifically,
  $H^*(BSO)$ is a free algebra generated by the Pontryagin classes in degree
  $4k$ for every $k$.  When $n-m$ is odd, $H^*(BSO_{n-m})$ is generated by
  Pontryagin classes in degree $4k$ where $2k<n-m$; when $n-m$ is even, it is
  generated by these plus an Euler class in degree $n-m$ whose square is the top
  Pontryagin class.

  \subsubsection*{The smooth case in odd codimension}
  In the case when $n-m$ is odd, whether a lift exists in \eqref{eqn:BSO}, and
  therefore the immersibility of $M$ in $\mathbb{R}^n$, can be determined via
  the following algorithm.  Let $M$ be given to us by a $C^1$ triangulation
  immersed in $\mathbb R^N$; without loss of generality we can take
  $N \geq 2m+1$.  We can construct the classifying map
  $f:M \to \Gr_m(\mathbb R^N)$ as described in \S\ref{S:Grass}.

  Now we determine whether there is a lift of $f$, using the special properties
  of the relative Postnikov tower of the map $p:BSO_{n-m} \to BSO$.  After
  rationalization, $p$ becomes the inclusion
  \[\prod_{k=1}^{2(n-m)} K(\mathbb Q,4k) \to \prod_{k=1}^\infty K(\mathbb Q,4k).\]
  This means that a lift of $f$ exists if and only if the $k$th Pontryagin
  classes of $M$ are zero, for $2(n-m) < 4k \leq m$, and some finite
  obstructions can be resolved.  Moreover, this lift is rationally unique.

  We start by computing the Pontryagin classes of $M$.  If there is a nonzero
  class above dimension $2(n-m)$, then a lift does not exist.  Otherwise, we
  must compute representatives of the finitely many homotopy classes of maps
  $g:M \to BSO_{n-m}$ which have the same Pontryagin classes.  Then, for each of
  these representatives, we must test whether $p \circ g$ is homotopic to the
  classifying map of the tangent bundle of $M$.

  As approximations to $BSO_{n-m}$ and $BSO$, we use $\Gr_{m+N'}(\mathbb R^{n+N'})$
  and $\Gr_{m+N'}(\mathbb R^{N+N'})$, for $N'$ sufficiently large.  As explained
  in \S\ref{S:Grass}, we can construct the stabilization maps and therefore have
  a map $M \to \Gr_{m+N'}(\mathbb R^{N+N'})$, which by abuse of notation we also
  call $f$.  The map $p$ can be constructed similarly.

  It remains to explain in detail how to construct the finite set of
  representatives.  We first compute the Postnikov tower of the space
  $\Gr_{m+N'}(\mathbb R^{n+N'})$ up to dimension $m$, obtaining Postnikov stages
  and maps $P_k \xrightarrow{p_k} P_{k-1}$.  Note that a map $M \to P_{4k}$
  contains information about Pontryagin classes up to the $k$th.  At the end of
  the induction, we will have maps $g:M \to P_m$; homotopy classes of such maps
  are in bijection with $[M,BSO_{n-m}]$.

  Suppose now that we have computed the set $\Sigma_{k-1}$ of candidate maps
  $M \to P_{k-1}$.  For each map in $\Sigma_{k-1}$, we first decide whether it
  lifts to $P_k$ and if it does, compute a lift $f_k:M \to P_k$.  Now if $k$ is
  not a multiple of $4$ or $k>2(n-m)$, then the obstruction group
  $H^k(M;\pi_k(\Gr_{m+N'}(\mathbb R^{N+N'})))$ which determines the set of lifts
  is finite, and we can compute a representative of each homotopy class of
  lifts.

  If $k$ is a multiple of $4$ and $k \leq 2(n-m)$, then there are an infinite
  number of lifts, and we must restrict to those which have the desired
  $(k/4)$th Pontryagin class.  Since $BSO_{n-m}$ is rationally a product,
  \[H^k(P_k;\mathbb Q) \cong H^k(\Gr_{m+N'}(\mathbb R^{N+N'})),\]
  and by computing the induced map on cohomology we can find the Pontryagin
  class $\eta \in H^k(P_k;\mathbb Q)$.  Now we state a lemma:
  \begin{lem}
    There is a homomorphism $\ph:H^k(M;\pi) \to H^k(M;\mathbb Q)$ such that if
    $f_k,f_k':M \to P_k$ are lifts of a map $f_{k-1}:M \to P_{k-1}$ such that the
    obstruction to homotoping them over $f_{k-1}$ is $\omega \in H^k(M;\pi)$,
    then
    \[(f_k')^*\eta=f_k^*\eta+\ph(\omega).\]
  \end{lem}
  Assuming the lemma, we can compute the values of the homomorphism $\ph$ on the
  generators of $H^k(M;\pi)$ by computing $f_k^*\eta$ as well as the pullback
  along those maps $f_k'$ for which the obstruction to homotoping $f_k$ to
  $f_k'$ is one of the generators of $H^k(M;\pi)$.  From this, we can now
  compute the finite set of elements of $H^k(M;\pi)$ such that the corresponding
  lifts of $f_{k-1}$ have the desired Pontryagin class, and the corresponding
  finite set of lifts.  This completes the inductive step.
  \begin{proof}[Proof of the lemma.]
    The homomorphism $\ph$ is induced by a map $\pi \to \mathbb Q$, which in
    turn is induced by the composition
    \[K(\pi,n) \xrightarrow{i} P_k \xrightarrow{\eta} K(\mathbb Q,n),\]
    where $i$ is the inclusion of the fiber of the fibration $P_k \to P_{k-1}$
    and $\eta$ is the classifying map of the Pontryagin class.  From this the
    lemma can be proved by explicit computation.  Suppose we have a
    triangulation of $M$ and two lifts $f_k$ and $f_k'$ of $f_{k-1}$ which
    coincide on $M^{k-1}$; this can be achieved by a homotopy without loss of
    generality.  Then one has that
    \[(f_k')^*\eta=f_k^*\eta+[(\eta \circ i)_*w]\]
    where $w \in C^k(M;\pi)$ is the obstruction cochain.
  \end{proof}

  \subsubsection*{The smooth case in the metastable range}
  When $2n \geq 3m+1$, all Pontryagin classes in the relevant range are zero.
  However, when $n-m$ is even, there may be a nonzero Euler class in degree
  $n-m$, whose square is always zero.  This is the only infinite-order homotopy
  group of the fiber of the map $BSO_{n-m} \to BSO$ below dimension $n$.
  Moreover, this map is $(n-m)$-connected.  To show that the resulting lifting
  problem is decidable, we can use the results of Vok\v r\'inek \cite{Vok}, who
  shows that a lifting problem through a $k$-connected fiber is decidable if the
  only infinite-dimensional homotopy groups of this fiber are of dimensions
  $<2k$.  As in the odd-dimensional case, we can approximate the map
  $BSO_{n-m} \to BSO$ by maps between finite-dimensional Grassmannians.

  \subsubsection*{The smooth case in even codimension}
  Now suppose that $n-m$ is even; write $c=n-m$.  We will show that
  immersibility is undecidable in this situation if $m \geq 4c$.  In order to do
  this, we first show that the lifting problem
  \[\xymatrix{
    & BSO_c \ar[d] \\
    X \ar[r]^-f \ar@{-->}[ru] & BSO
  }\]
  is undecidable for general $2c$-complexes $X$ and maps $f:X \to BSO$.
  We prove this by a method used in \cite{CKMVW}, reducing an algebraic problem
  which is undecidable by \cite[Lemma 2.1]{CKMVW} to a question about lifts.
  The undecidable problem in question is a special case of Hilbert's 10th
  problem: determining the existence of an integer solution to a system of
  equations each of the form
  \begin{equation} \label{sum-of-squares}
    \sum_{1 \leq i<j \leq r} a_{ij}^{(k)} x_ix_j=b_k,
  \end{equation}
  where $x_1,\ldots,x_r$ are variables and $b_k$ and $a_{ij}^{(k)}$ are
  coefficients.

  The idea is to build a CW complex $X$ with a map $f:X \to BSO$ determined by
  the $b_k$, such that any lift of $f$ to $BSO_c$ determines an assignment of
  the variables $x_i$.  This is done by having the $b_k$ determine the
  $2c$-dimensional Pontryagin class of $f$; since in $BSO_c$, this Pontryagin
  class is the square of a $c$-dimensional Euler class, this forces us to find a
  corresponding $c$-dimensional class in $X$ which is the pullback of this Euler
  class under the lift.  The relationships between the pairings of these classes
  with $c$- and $2c$-dimensional homology classes will be given by
  \eqref{sum-of-squares}. The process is complicated by some finite-order
  phenomena; we now give the construction in detail.

  Let $\alpha \in \pi_{2c}(BSO) \cong \mathbb Z$ be a generator; we have
  \[\alpha^*p([S^{2c}])=n_1\]
  where $p \in H^{2c}(BSO)$ is the Pontryagin class in degree $2c$ and $n_1$ is
  some integer.  Denote the pullback of this Pontryagin class by
  $\tilde p \in H^{2c}(BSO_c)$.  Similarly, let $\beta \in \pi_c(BSO_c)$ be an
  element which (if $c$ is a multiple of $4$) pairs trivially with the
  Pontryagin class in degree $c$ and (subject to this restriction) has the
  smallest possible nontrivial pairing with the Euler class
  $\eta \in H^c(BSO_c)$.  Let $n_2$ be an integer such that
  \[\beta^*\eta([S^c])=n_2.\]
  Note that rationally, $\eta^2$ is a multiple of $\tilde p$, so there are
  integers $n_3,n_4 \neq 0$ such that
  \[n_3\eta^2=n_4\tilde p.\]
  Finally, let $n_5=|\pi_{2n-1}(BSO_c)|$, since this homotopy group is of finite
  order.

  Given a system of $s$ equations of the form \eqref{sum-of-squares}, we form a
  CW complex $X$ as follows.  We take the wedge of $r$ copies of $S^c$, which we
  label $S^c_1,\ldots,S^c_r$, and attach $s$ $2c$-cells, the $k$th cell $e_k$ via
  an attaching map whose homotopy class is a linear combination of Whitehead
  products
  \[\sum_{1 \leq i<j \leq r} n_1n_4n_5a_{ij}^{(k)}[\id_i,\id_j],\]
  where $\id_i$ is the inclusion map of $S^c_i$.  We fix a map $f:X \to BSO$ by
  taking the $c$-cells to the basepoint and the $k$th $2c$-cell to a
  representative of $2n_2^2n_3n_5b_k\alpha$.

  Then for any homotopy lift $\tilde f:X \to BSO_c$ of $f$, we have
  \begin{equation} \label{whitehead}
    n_4\tilde f^*\tilde p(e_k)=n_3\tilde f^*\eta^2(e_k)=n_1n_3n_4n_5\sum_{1 \leq i<j \leq r} 2a_{ij}^{(k)}\tilde f^*\eta([S^c_i])\tilde f^*\eta([S^c_j]),
  \end{equation}
  where the last equation can be obtained by analyzing, for each $i$ and $j$,
  the maps $X \to S^c_i \times S^c_j$ which send all the other spheres to a
  point.  Moreover, if $c$ is a multiple of $4$, the pullback of the degree $c$
  Pontryagin class along $\tilde f$ is zero.  Therefore, the numbers
  $x_i=n_2^{-1}\tilde f^*\eta([S^c_i])$ are integers, and \eqref{whitehead}
  reduces to
  \[2n_1n_2^2n_3n_4n_5b_k=2n_1n_2^2n_3n_4n_5\sum_{1 \leq i<j \leq r} a_{ij}^{(k)}x_ix_j,\]
  showing that \eqref{sum-of-squares} is satisfied.

  Conversely, if we choose $x_1,\ldots,x_r$ satisfying \eqref{sum-of-squares},
  then the map $\bigvee_{i=1}^r S^c \to BSO_c$ which maps $S^c_i$ via a
  representative of $x_i\beta$ extends to a map $\tilde f:X \to BSO_c$ because
  the attaching map of every $2c$-cell is divisible by $n_5$ and therefore is
  zero in $\pi_{2c-1}(BSO_c)$.  The projection of $\tilde f$ to $BSO$ is then
  nullhomotopic on each $c$-cell and maps each $2c$-cell $e_k$ to $BSO$ via a
  representative of $2n_2^2n_3n_5b_k\alpha$.

  It remains to show that one can construct a manifold whose homotopy type and
  Pontryagin classes determine any such system.  This can be done, at the cost
  of some increase in dimension; our examples are of dimension at least $4c$,
  which is probably not optimal.

  Such manifolds exist by an argument of Wall \cite[\S5]{Wall}, who shows that
  for any $2c$-complex $X$ and map $f:X \to BSO$, and any $q \geq 4c$, there
  is a corresponding $(q+1)$-dimensional \emph{thickening} of $X$, i.e.~a
  manifold with boundary $M$ homotopy equivalent to $X$ such that the
  classifying map of its tangent bundle is homotopic to $f$.  Moreover, the pair
  $(M,\partial M)$ is $(q-2c)$-connected and any extra topology of $\partial M$
  is sent to zero by the classifying map.  Thus $\partial M$ is a closed
  $q$-manifold which immerses in $\mathbb{R}^{q+c}$ if and only if the system of
  equations above has a solution.

  Now suppose there were an algorithm to decide smooth immersibility of
  $q$-manifolds in $\mathbb{R}^{q+c}$ for some fixed even $c$ and $q \geq 4c$.
  Then given a system of equations, we could iterate over smooth closed
  $q$-manifolds $M$ and bases for $H^c(M)$ until we find one with the right
  cohomology algebra and classifying map.  This search terminates since Wall
  guarantees the existence of such a manifold.  Then we could decide whether the
  system has a solution using our solution to the immersibility problem.  Thus
  immersibility cannot be decidable.

  \subsubsection*{The PL case in codimension $\geq 3$}
  Denote the universal cover of a space $X$ by $\widetilde X$.  In the PL case,
  similarly to the smooth case, the unstable lifting problem reduces to the
  stable problem
  \[\xymatrix{
    & \widetilde{BPL}_{n-m} \ar[d] \\
    M \ar[r]^-\kappa \ar@{-->}[ru] & \widetilde{BPL}.
  }\]
  Moreover, when $n-m \geq 3$, the diagram
  \[\xymatrix{
    BPL_{n-m} \ar[r] \ar[d] & BPL \ar[d] \\
    BG_{n-m} \ar[r] & BG,
  }\]
  where $BG$ is the classifying space of spherical fibrations, is a homotopy
  pullback square \cite[p.~123]{WallBk}.  Thus, equivalently, we must solve the
  lifting problem
  \[\xymatrix{
    & \widetilde{BG}_{n-m} \ar[d] \\
    M \ar[r]^-\kappa \ar@{-->}[ru] & \widetilde{BG}.
  }\]
  The argument in \S2 shows that the only infinite homotopy group of
  $G_{n-m}$ is
  \[\left\{\begin{array}{l l}
  \pi_{n-m-1} & \text{when $n-m$ is even} \\
  \pi_{2(n-m)-3} & \text{when $n-m$ is odd,}
  \end{array}\right.\]
  and therefore $BG_{n-m}$ only has a single infinite homotopy group in dimension
  $n-m$ or $2(n-m)-2$, depending on parity.  Moreover, in both cases the map
  $\widetilde{BG}_{n-m} \to \widetilde{BG}$ is $(n-m-2)$-connected, by the
  stability of homotopy groups of spheres.

  Thus to decide immersibility we can use the following algorithm.  Suppose $M$
  is given to us as a simplicial complex.  By abuse of notation we refer to
  $\widetilde{BG}_n$ when we really mean the approximations constructed in
  \S\ref{S:BGn}.
  \begin{enumerate}
  \item We embed the simplicial complex linearly in $\mathbb{R}^N$, for some
    large $N$.
  \item This gives us a map $M \to \widetilde{BG}_N$ which can be computed by
    Lemma \ref{lem:BGn}(\ref{BGniii}).
  \item Decide whether the map lifts to a map $M \to \widetilde{BG}_{n-m}$.  In
    the even case, this can be done by the aforementioned work of Vok\v r\'inek
    \cite{Vok}, since the only infinite obstruction is below twice the
    connectivity of the map $\widetilde{BG}_{n-m} \to \widetilde{BG}$.  In the
    odd-dimensional case, we can split the work into two steps:
    \begin{itemize}
    \item Compute all possible lifts to the $(2(n-m)-3)$rd stage of the relative
      Postnikov tower of $\widetilde{BG}_{n-m} \to \widetilde{BG}$.  This can be
      done since all the obstructions are finite.
    \item For each lift computed, use the algorithm of Vok\v r\'inek to decide
      whether it can be extended to $\widetilde{BG}_{n-m}$.
    \end{itemize}
  \end{enumerate}

  \subsubsection*{PL immersions in codimension 2}
  In codimension 2, there are two somewhat different things we may mean by PL
  immersion: locally flat immersion, in which the link of every vertex is
  unknotted, and immersion which is not necessarily locally flat.

  A PL manifold $M$ has a locally flat immersion in codimension 2 if and only if
  it has a smoothing which immerses smoothly in codimension 2.  This is because
  by the fundamental theorem of smoothing theory \cite[Part II]{HM}, $M$ is
  smoothable if and only if the classifying map $M \to \widetilde{BPL}$ of the
  stable tangent bundle lifts to $BSO$; but immersibility is equivalent to the
  existence of a further lift
  \[\xymatrix{
     & BSO_2 \cong \widetilde{BPL}_2 \ar[d] \\
     & BSO \ar[d] \\
    M \ar[r] \ar@{-->}[ruu] & \widetilde{BPL}.
  }\]
  Moreover, the homotopy fiber $PL/O$ has finite homotopy groups, so the
  rational obstructions discussed above are the same in the PL case as in the
  smooth case.  Therefore, this problem is undecidable for $\dim M \geq 8$ by
  the same argument as above: the examples we produced are PL immersible if and
  only if they are smoothly immersible.

  The case of immersions which are not necessarily locally flat was studied by
  Cappell and Shaneson \cite{CaSh76,CaSh73}.  Such immersions are
  classified by maps to a space $BSRN_2$.  Unlike in the higher codimension
  case, the diagram
  \[\xymatrix{
    BSRN_2 \ar[r] \ar[d] & \widetilde{BG_2} \ar[d] \\
    \widetilde{BPL} \ar[r] & \widetilde{BG}
  }\]
  is not a homotopy pullback square, but the map from $BSRN_2$ to the pullback
  splits up to homotopy.  Therefore it is again sufficient to solve the lifting
  problem from $\widetilde{BG}$ to $\widetilde{BG_2}$.

  \subsubsection*{Codimension 1}
  In codimension one, the lifting problem above, and therefore the question of
  smooth immersibility, boils down to whether the suspension of the tangent
  bundle is trivial, that is, whether the composition
  \[M \to BSO_m \to BSO_{m+1}\]
  is nullhomotopic.  Once this composition is given as an explicit map, whether
  it is nullhomotopic is a decidable question by Proposition
  \ref{off-the-shelf}\eqref{homotopic} (due to \cite{FiVo}).

  The oriented PL case is formally identical: one needs to determine whether the
  map $M \to \widetilde{BPL}_{m+1}$ induced by the tangent bundle, or
  equivalently the map $M \to \widetilde{BPL}$, is trivial.  However, up until
  now we have gotten away with only studying maps to $\widetilde{BG}$, and we
  have neither an explicit finite-type model for $\widetilde{BPL}$ nor a way of
  constructing the map.  One way of getting around this would be to first
  determine whether the map to $\widetilde{BG}$ is trivial; if it is, then there
  is an induced map to $G/PL$ which must also be trivial.  To determine its
  triviality, we would need to compute the Pontryagin and Kervaire classes of
  $M$ from its combinatorial structure.  While it is known in principle that
  local combinatorial formulas can be used to compute these classes \cite{LR},
  no explicitly computable such formulas are known, except for the first
  Pontryagin class computed by Gaifullin; see the survey article \cite{Gai}.
  For example, the well-known construction of rational Pontryagin classes by
  Gelfand and MacPherson \cite{GeMP} uses either a smooth structure or an
  additional piece of data replacing it.

  The path we take uses smoothing theory.  As in the codimension 2 case, if the
  classifying map $M \to \widetilde{BPL}$ is trivial, $M$ admits a smoothing
  (that is, a smooth structure which is compatible with the PL structure) which
  immerses smoothly in $\mathbb{R}^{m+1}$.  Thus it is enough to construct all
  possible smoothings of $M$ (finitely many, and perhaps none); then we can use
  the smooth algorithm to determine whether one of them immerses.  This
  construction is given in \S\ref{S:smoothings}.

  \subsubsection*{Non-orientable manifolds}
  In this case constructing an immersion is equivalent to constructing a
  $\mathbb{Z}/2\mathbb{Z}$-invariant immersion of the oriented double cover.
  In other words, we must do what we did above but in a way that respects the
  natural free $\mathbb{Z}/2\mathbb{Z}$-action on each of the classifying
  spaces.  This action is easy to encode computationally; moreover, as pointed
  out by Vok\v r\'inek \cite[\S5]{Vok} and elaborated in \cite{CKV}, the
  relevant homotopy theory can be done as easily as in the non-equivariant case.
\end{proof}

\section{Applications to embeddings}

\subsection{Immersions which extend to embeddings} \label{S:closed}

The following is a well-known fact, noted for example in \cite{Massey}.
\begin{lem}
  The normal bundle to an embedded smooth closed oriented submanifold
  $M^m \subseteq \mathbb{R}^n$ always has vanishing Euler class.
\end{lem}
\begin{proof}
  Consider the diagram
  \[\xymatrix{
    H^{n-m}(\mathbb{R}^n,\mathbb{R}^n \setminus M) \ar[r] \ar[d] &
    H^{n-m}(\mathbb{R}^n) \ar[d] \\
    H^{n-m}(\nu_M,\nu_M \setminus M) \ar[r]^-{(*)} & H^{n-m}(\nu_M) \ar[r] &
    H^{n-m}(M).
  }\]
  The Euler class is the image of the generator of
  $H^{n-m}(\nu_M,\nu_M \setminus M)$ along the bottom row.  The left vertical
  arrow is an isomorphism by excision.  Since $H^{n-m}(\mathbb{R}^n)=0$, the
  arrow labeled $(*)$ is zero.
\end{proof}
This means that if $M$ is closed and oriented, an immersion of $M$ can only be
regularly homotopic to an embedding if it has zero Euler class.  Unlike the
existence of an immersion in general, the existence of such an immersion is
decidable via the same algorithm as in odd codimension: test whether all
Pontryagin classes in degrees $2(n-m) \leq 4i \leq 2m$ are zero, and then
resolve the remaining finite-order questions.

In other words, while it may well be that the embeddability of closed smooth
manifolds in $\mathbb{R}^n$ is undecidable outside the metastable range, this
cannot be a result of immersion theory.

\subsection{Embeddability is undecidable}

\begin{thm} \label{thm:emb}
  Whenever $n-m$ is even and $11m \geq 10n+1$, the embeddability of a smooth
  $m$-manifold with boundary in $\mathbb{R}^n$ is undecidable.
\end{thm}
We note that the method used here depends both on using the smooth category and
on allowing the manifold to have boundary.
\begin{proof}
  We reduce this statement to Theorem \ref{decundec}(\ref{undec}).  We note
  first that by the stability property discussed above, when $n \geq m+2$, an
  $m$-manifold $M$ immerses smoothly in $\mathbb{R}^n$ if and only if
  $M \times D^k$ immerses smoothly in $\mathbb{R}^{n+k}$.

  In general position, the self-intersection of an immersion $f:M \to N$ is a
  $(2m-n)$-dimensional CW complex.  If we stabilize by crossing with
  $\mathbb{R}^k$ for $k \geq 4m-2n+1$, then this complex always has an embedding
  in $\mathbb{R}^k$; therefore the immersion
  \[f \times \id:M \times D^k \to \mathbb{R}^{n+k}\]
  can be deformed to an embedding, by pushing a neighborhood of the
  self-intersection off itself in the $\mathbb{R}^k$ direction.  Conversely, if
  $M$ does not immerse in $\mathbb{R}^n$, then $M \times D^k$ does not embed in
  $\mathbb{R}^{n+k}$.

  If $m=4c$ and $n=5c$, then we can choose $k=6c+1$.  In other words, it is
  undecidable whether a $(10c+1)$-manifold embeds into $\mathbb{R}^{11c+1}$ when
  $c$ is even.
\end{proof}

\section{Computing all smoothings of a PL manifold} \label{S:smoothings}

In this section we sketch an algorithm which, given a triangulation of a PL
manifold $M^m$, computes a set of $C^1$ manifolds which contains at least one
(but usually many) representatives of each diffeomorphism type of smoothing of
$M$.  The manifolds are given in the form of a subdivision of the original
triangulation equipped with a simplexwise polynomial immersion to some
$\mathbb R^N$, as described in \S\ref{S:smooth}.  Of course, if $M$ is not
smoothable, the algorithm yields the empty set.

The algorithm naturally splits into two pieces: some computations related to the
groups $\Theta_k$ of exotic spheres in dimensions $k \leq m$, and an inductive
procedure which relies on those computations.

\subsubsection*{The inductive procedure} We start by fixing a subdivision of the
$m$-simplex such that for each $0 \leq k \leq m$ there is a pure $m$-dimensional
subcomplex $U_k \subset \Delta^m$ such that:
\begin{itemize}
\item $U_k$ deformation retracts to the $k$-skeleton of $\Delta^m$.
\item $U_k$ is invariant under permutations of the vertices of $\Delta^m$.
\item $U_k$ is contained in the interior of $U_{k+1}$.
\end{itemize}
An example of such a subdivision of $\Delta^2$ is illustrated in
Figure~\ref{thefig}.  We write
\[V_k=\bigcup_{\sigma \in M} U_k.\]

\begin{figure}
  \begin{tikzpicture}[scale=1]
    \coordinate (a) at (330:3);
    \coordinate (b) at (210:3);
    \coordinate (c) at (90:3);
    \draw[clip] (a) -- (b) -- (c) -- cycle;
    \filldraw[fill=gray!20] (a)--(b)--(c)--cycle;
    \filldraw[very thick,fill=gray!50] (a) circle [radius=1]; 
    \filldraw[very thick,fill=gray!50] (b) circle [radius=1]; 
    \filldraw[very thick,fill=gray!50] (c) circle [radius=1]; 
    \filldraw[very thick,fill=white] (0,0) circle [radius=1];
    \draw (a) -- (barycentric cs:b=0.5,c=0.5) (b) -- (barycentric cs:c=0.5,a=0.5) (c) -- (barycentric cs:a=0.5,b=0.5);
    \draw (330:1) -- (barycentric cs:b=0.19245,a=0.80755) -- (270:1) -- (barycentric cs:a=0.19245,b=0.80755) -- (210:1) -- (barycentric cs:c=0.19245,b=0.80755) -- (150:1) -- (barycentric cs:b=0.19245,c=0.80755) -- (90:1) -- (barycentric cs:a=0.19245,c=0.80755) -- (30:1) -- (barycentric cs:c=0.19245,a=0.80755) -- cycle;
  \end{tikzpicture}
  \caption{A subdivision of $\Delta^2$ satisfying the required conditions.  The
    sets $U_0$ and $U_1$ are highlighted in shades of gray.  We have drawn the
    simplices curved to suggest how $U_0$ and $U_1$ will look as subsets of a
    smooth manifold with boundary.} \label{thefig}
\end{figure}

We then construct all possible smoothings of $M$ by induction on $k$: first we
construct all possible smoothings of $V_0$, then extend them to $V_1$ in all
possible ways, and so on.  At each step each representative will be encoded via
a smooth map from a further subdivision to $\mathbb{R}^N$, for some fixed
$N \geq 2m+1$.  Once we have extended to $V_m=M$, we have generated
representatives of all possible smoothings.  The base case is clear: since there
is a unique smooth structure on a compact PL disk, we can choose an arbitrary
smooth map $V_0 \to \mathbb R^N$.

Now suppose we have defined a smoothing $f_{k-1}:V_{k-1} \to \mathbb R^N$.  Then
the $k$th step of the induction proceeds as follows, for every $k$-simplex
$\sigma$ of $M$:
\begin{enumerate}
\item Determine whether the map on $\partial V_{k-1} \cap \sigma$ is
  diffeomorphic to the standard $(k-1)$-sphere.  If it isn't, then the smoothing
  does not extend to $\sigma$.
\item If the smoothing extends, we extend it over $\sigma$ by iterating over all
  possible piecewise polynomial smooth maps until we find one that works.  The
  result is a smoothing of $V_{k-1} \cup M^k$: that is, it is both a $C^1$
  embedding of an $m$-manifold with boundary when restricted to $V_{k-1}$ and a
  $C^1$ embedding of a $k$-manifold when restricted to the interior of each
  $k$-simplex.
\item For every exotic $k$-sphere, we surger in a $D^k$ which modifies the
  smoothing on $\sigma \setminus V_{k-1}$ by that $k$-sphere.  That is, we cut
  out a (subdivided) simplex of $\sigma \setminus V_{k-1}$ and glue in an exotic
  sphere missing a simplex, with a cylindrical ``neck'' making a $C^1$
  connection between them.  Thinking of this as a map from $\sigma$ entails a
  further subdivision.
\item For every smoothing of $V_{k-1} \cup M^k$ thus generated, the stability of
  smoothing theory guarantees that there is a unique smooth structure extending
  it over $V_k$, which deformation retracts to $V_{k-1} \cup M^k$.  We construct
  such an extension by exhaustive search.
\end{enumerate}

\subsubsection*{Algorithms for exotic spheres}
To give detailed instructions for steps (1) and (3), we must describe algorithms
for constructing and classifying exotic spheres.  In every dimension $k$, the
exotic spheres are classified by a finite abelian group $\Theta_k$ whose group
operation is connected sum.  To perform steps (1) and (3), it would be enough to
have an algorithm which, given a smooth manifold PL homeomorphic to the sphere,
computes the corresponding element of $\Theta_k$.  For step (1), we must simply
test whether the element of $\Theta_k$ induced by
$f_{k-1}|\partial V_{k-1} \cap \sigma$ is zero.  For step (3), we can generate all
the exotic spheres by iterating over all possible piecewise polynomial smooth
maps from subdivisions of $\partial\Delta^{k+1}$ to $\mathbb{R}^N$ until we find
representatives for every element of $\Theta_k$.  Then we can get the desired
disks by cutting out a $k$-simplex from each of these.

In fact, we find something slightly weaker.  To analyze the group $\Theta_k$, we
look at the original paper of Kervaire and Milnor \cite{KerMil} where it is
defined, as well as an expository paper of Levine \cite{Levine} which fills in
certain details developed later.  It turns out that $\Theta_k$ naturally
fits into an exact sequence, whose terms we will define later:
\[0 \to bP_{k+1} \to \Theta_k \xrightarrow{\psi}
\coker\bigl(\pi_k(SO_{k+1}) \xrightarrow{J_k} \pi_{2k+1}(S^{k+1})\bigr)
\xrightarrow{\phi} P_k.\]
We sketch algorithms which, given a smooth manifold PL homeomorphic to the
sphere,
\begin{itemize}
\item[($*$)] compute the corresponding element of $\Theta_k/bP_{k+1}$;
\item[($\dagger$)] if this element is zero, compute the corresponding element of
  $bP_{k+1}$.
\end{itemize}
This is clearly enough for step (1); for step (3), if we generate
representatives of all elements of $\Theta_k/bP_{k+1}$ and all elements of
$bP_{k+1}$, we can generate representatives of all elements of $\Theta_k$ by
taking connect sums.

We now discuss the terms of the exact sequence above:
\begin{itemize}
\item The group $\displaystyle P_k=\left\{\begin{array}{l l}0 & k\text{ odd}\\
  \mathbb{Z}/2\mathbb{Z} & k \equiv 2 \mod 4\\
  \mathbb{Z} & k \equiv 0 \mod 4.\end{array}\right.$
\item The map $\phi$ sends a smooth map $f:S^{2k+1} \to S^{k+1}$ to the Kervaire
  invariant (if $k \equiv 2 \mod 4$) or $1/8$ times the signature (if
  $k \equiv 0 \mod 4$) of the preimage of a regular point.
\item The map $J_k$ is the usual $J$-homomorphism, defined as follows.  An
  element of $\pi_r(SO_q)$ can be interpreted as a map
  $S^r \times S^{q-1} \to S^{q-1}$.  This in turn induces a map from the join
  $S^r * S^{q-1} \cong S^{r+q}$ to the suspension of $S^{q-1}$, that is, $S^q$.
\item The group $bP_{k+1}$ is a certain finite quotient of $P_{k+1}$.  In the
  nontrivial case $k+1=2r$, this has order which divides
  \[2^{2r-1} \cdot (2^{2r-1}-1) \cdot \text{numerator}(B_r/r),\]
  where $2r=k+1$ and $B_r$ is the $r$th Bernoulli number.
\item The map $\psi$ is constructed as follows.  Every smooth homotopy sphere
  $\Sigma$ is stably parallelizable.  This means that given an embedding
  $\Sigma \hookrightarrow S^{2k+1}$, one can construct a trivialization of the
  normal bundle and use the Pontryagin--Thom construction to give a map
  $S^{2k+1} \to S^{k+1}$.  This depends on the choice of trivialization, and the
  indeterminacy is exactly the image of the $J$-homomorphism.
\item Finally, an isomorphism between $\ker\psi$ and $bP_{k+1}$ is given as
  follows.  If $\Sigma \in \ker\psi$, then $\Sigma$ is framed nullcobordant.
  Then the corresponding element in $P_{k+1}$ is given by the Kervaire invariant
  (when $k+1$ is odd) or $1/8$ signature (rel boundary, when $k+1$ is even) of a
  nullcobordism with parallelizable normal bundle; this has an indeterminacy
  which induces the quotient map $b$.
\end{itemize}
It remains to show that all of these elements can be computed.

The signature and Kervaire invariant are cohomological notions and so are
unproblematic to compute from a triangulation.

A generator for $\pi_k(SO_{k+1})$ can be constructed explicitly as a simplicial
map by the main theorem of \cite{FFWZ}.  Then a corresponding simplicial map
$S^k \times S^{k} \to S^{k}$ can be constructed by induction on skeleta of
$S^r$.  Finally, the Hopf construction of a map from the join to the suspension
is clearly algorithmic.  This gives an algorithm for determining the image of
the $J$-homomorphism.

By results of \cite{CKMVW2}, $\pi_{2k+1}(S^{k+1})$ is fully effective: that is, we
can compute a set of generators and find the combination of generators which
represents the homotopy class of a given simplicial map.  In particular, this
allows us to compute the cokernel of the $J$-homomorphism.

The main remaining obstacle is implementing the map $\psi$.  We can do this by
explicitly constructing, given a smooth homotopy $k$-sphere $\Sigma$, a map
$S^{2k+1} \to S^{k+1}$ realizing the Pontryagin--Thom construction.

If $\Sigma$ is specified by a $C^1$ piecewise polynomial embedding
$f:X \to \mathbb R^N$, where $X$ is a simplicial complex homeomorphic to $S^k$,
then a smooth structure on $\Sigma \times D^{k+1}$ (with each product of
simplices triangulated in a standard way) is given by
\[f \times i:\Sigma \times D^{k+1} \to \mathbb R^{N+k+1}.\]
We then iterate through simplexwise polynomial maps with rational coefficients
from subdivisions of $\Sigma \times D^{k+1}$ to $[0,1]^{2k+1}$ until we find
a map $g:\Sigma \times D^{k+1} \to [0,1]^{2k+1}$ that is $C^1$ and injective.
Both these conditions can be checked.  One checks $C^1$ by checking that
derivatives match between neighboring simplices.  To check injectivity, one
shows that every simplex is injective, the images of any two non-adjacent
simplices are disjoint, and the intersection of the images of adjacent simplices
is the image of their intersection.  Each of these can be expressed as a
sentence in the language of the reals and can be decided by the
Tarski--Seidenberg theorem; see e.g.~\cite[Ch.~11]{BPR} for relatively practical
algorithms.

Now, let $p:\Sigma \times D^{k+1} \to S^{k+1}$ be the map which projects to the
$D^{k+1}$ factor and then collapses the boundary.  This map is easily made
simplicial.  The desired map $S^{2k+1} \to S^{k+1}$ is obtained from a map
$u:[0,1]^{2k+1} \to S^{k+1}$ which maps the image of $g$ to $S^{k+1}$ via
$p \circ g^{-1}$ and everything outside the image of $g$ to the base point.
Although $g^{-1}$ is not piecewise polynomial, we can approximate its values to
arbitrary precision and bound its Lipschitz constant.  By Lemma \ref{simp-approx}, this suffices to construct a simplicial approximation to $u$
from a sufficiently fine subdivision of $[0,1]^{2k+1}$.  From this we can
compute its homotopy class in $\pi_{2k+1}(S^{k+1})$, and therefore the image of
$\Sigma$ under $\psi$.

Moreover, given a map in $\ker\psi$, we can find a framed nullcobordism
by iterating over all candidate smooth manifolds $\Upsilon$ whose boundary is
$\Sigma$ and smooth embeddings $\Upsilon \times D^{k+1} \to D^{2k+2}$ extending
$g$.  Thus, given a homotopy sphere, we can assign it either to a nonzero
element of $\Theta_k/bP_{k+1}$ or an element of $bP_{k+1}$.  By iterating over all
smooth triangulations corresponding to barycentric subdivisions of
$\partial\Delta^{k+1}$, we eventually generate representatives for all the
elements of both the subgroup and the quotient group.

\bibliographystyle{amsalpha}
\bibliography{comput}
\end{document}